\documentclass[11pt]{amsart}
\usepackage{color, tikz, float}
\usetikzlibrary{arrows, decorations.markings}

\newcommand{\bt}{\begin{Theorem}}
\newcommand{\et}{\end{Theorem}}
\newcommand{\bi}{\begin{itemize}}
\newcommand{\ei}{\end{itemize}}
\newcommand{\bea}{\begin{eqnarray}}
\newcommand{\ba}{\begin{array}}
\newcommand{\eea}{\end{eqnarray}}
\newcommand{\ea}{\end{array}}

\newtheorem{Definition}{Definition}[section]
\newtheorem{Theorem}[Definition]{Theorem}
\newtheorem{Lemma}[Definition]{Lemma}

\newtheorem{Proposition}[Definition]{Proposition}

\newtheorem{Remark}[Definition]{Remark}

\newtheorem*{theoremA*}{Theorem A}
\newtheorem*{theoremB*}{Theorem B}
\newtheorem*{theoremC*}{Theorem C}
\newtheorem*{Proofofmainthm*}{Proof of main theorem}
\newcommand{\be}{\begin{equation}}
\newcommand{\ee}{\end{equation}}
\newcommand{\bvb}{\begin{verbatim}}

\newcommand{\evb}{\end{verbatim}}

\newcommand{\wtilde}{\widetilde}%
\newcommand{\R}{\mathbb R}%
\newcommand{\C}{\mathbb C}%
\newcommand{\Z}{\mathbb Z}%
\newcommand{\N}{\mathbb N}%
\newcommand{\x}{\mathfrak X}%

\newcommand{\cL }{\mbox{$ \mathcal L $}}%
\renewcommand{\Re}{\mbox{Re }}
\begin{document}
\baselineskip16pt

\author[Pratyoosh Kumar and Sumit Kumar Rano]{Pratyoosh Kumar and Sumit Kumar Rano }
\address{Department of Mathematics, Indian Institute of Technology Guwahati, Guwahati, 781039, India.
E-mail: pratyoosh@iitg.ac.in and s.rano@iitg.ac.in}

\title[Dynamics of semigroups on Homogeneous Trees]
{Dynamics of semigroups generated by analytic functions of the Laplacian on Homogeneous Trees}
\thanks{Second Author is supported by Institute fellowships of IIT Guwahati.}

\begin{abstract} Let $f$ be a non-constant complex-valued analytic function defined on a connected, open set containing the $L^p$-spectrum of the Laplacian $\mathcal L$ on a homogeneous tree. In this paper we give a necessary and sufficient condition for the semigroup $T(t)=e^{tf(\mathcal{L})}$ to be chaotic on $L^{p}$-spaces. We also study the chaotic dynamics of the semigroup $T(t)=e^{t(a\mathcal{L}+b)}$ separately and obtain the sharp range of $b$ for which $T(t)$ is chaotic on $L^{p}$-spaces. It includes some of the important semigroups, such as the heat semigroup and the Schr\"{o}dinger semigroup.

\end{abstract}
\maketitle
\section{Introduction}
A homogeneous tree $\mathfrak{X}$ of degree $q+1$ is a connected graph with no circuits such that every vertex is connected to $q+1$ other vertices. Henceforth we assume $q\geq 2$. We denote by $d(x,y)$ the natural distance between any two vertices $x$ and $y$, which is the number of edges joining them. The canonical Laplacian $\mathcal{L}$ on $\mathfrak{X}$ is defined by
$$\mathcal Lf(x)=f(x)-\frac{1}{q+1}\sum\limits_{y:d(x,y)=1}f(y).$$

Unlike many other spaces, $\mathcal{L}$ defines a bounded linear operator on the Lebesgue spaces $L^{p}(\mathfrak{X})$ for every $p\in[1,\infty]$. Let $\sigma_p(\cL)$ denote the $L^p$-spectrum of the Laplacian $\mathcal L$. Let $f$ be a non-constant complex holomorphic function defined on a connected open set containing $\sigma_{p}(\mathcal{L})$. Then by the usual Riesz functional calculus (see \cite[Page 261]{WR}), it follows that the semigroup
\be\label{semigroup}T(t)=e^{tf(\mathcal{L})}\quad\text{where }t\geq 0,\ee
is a bounded linear operator on $L^{p}(\mathfrak{X})$ for every $p\in[1,\infty]$.

 In this paper we study the chaotic dynamics of the semigroup $T(t)=e^{tf(\mathcal{L})}$ on $L^p(\x)$. Before stating our main results, we recall some basic definitions. For details we refer \cite{MP,D}. Let $X$ be a Banach space and $\mathcal B(X)$ be the space of all bounded linear operators from $X$ into itself. A \emph{semigroup} on $X$ is a map $T: [0,\infty)\to \mathcal B(X)$ such that $T(0)=I$ and $T(s+t)=T(s)T(t)$ for all $s,t\geq 0$. Further $T(t)$ is said to be \textit{hypercyclic} if there exists $x\in X$ such that $\{T(t)x:t\geq 0\}$ is dense in $X$. A point $x\in X$ is said to be \textit{periodic} for $T(t)$ if there exists $t>0$ such that $T(t)x=x$. The set of all periodic points will henceforth be denoted by $X_{per}$. The semigroup $T(t)$ is said to be \textit{chaotic} if it is hypercyclic and its set of periodic points is dense in $X$.

It follows from the definition of hypercyclicity that the existence of a hypercyclic semigroup on a Banach space implies that it is separable. Therefore
 it is obvious that $T(t)$ cannot be hypercyclic on $L^\infty (\x)$ and hence not chaotic on $L^\infty (\x)$.
For other values of $p$, we have the following results.
\begin{theoremA*}{\label{ThA}}
		Let $2<p<\infty$ and $T(t)=e^{tf(\mathcal L)}$ be a semigroup on $L^p(\x)$ as defined in (\ref{semigroup}). Then the following statements are equivalent.
	\begin{itemize}
		\item [(1)] $T(t)$ is chaotic on $L^{p}(\mathfrak{X})$.
		\item [(2)]$T(t)$ has a non-trivial periodic point, that is $L^{p}(\mathfrak{X})_{per}\neq\{0\}$.
		\item  [(3)]The set of periodic points of $T(t)$ is dense in $L^{p}(\mathfrak{X})$, that is $\overline{L^p(\x)}_{per}=L^{p}(\mathfrak{X})$.
	\end{itemize}
\end{theoremA*}

\begin{theoremB*} Let $1\leq p\leq 2$ and $T(t)$ be as in Theorem A. Then we have the following.
	\begin{itemize}
		\item[(1)] $T(t)$ has no non-trivial periodic point in $L^{p}(\mathfrak{X})$.
		\item[(2)] $T(t)$ is not hypercyclic on $L^{p}(\mathfrak{X})$.
		\end{itemize}
In particular $T(t)$ is not chaotic on $L^{p}(\mathfrak{X})$.
\end{theoremB*}
For $2<p<\infty$ we define $\delta_p=1/p-1/2$ and
\be\label{Phia}\Phi_{p}(a)=(1-\gamma(i\delta_{p}))\cdot((\Re a)^{2}+\tanh^{2}(\delta_{p}\log q)(\Im a)^{2})^{1/2}\ee
where $\gamma(z)=1-\frac{q^{1/2+iz}+q^{1/2-iz}}{q+1}$.
\begin{theoremC*}{\label{ThB}}
	Suppose that $T(t)=e^{t(a\mathcal{L}+b)}$, $t\geq 0$ where a is a non-zero complex number and $b$ is real. Let $2<p<\infty$. Then the following are equivalent.
	\begin{itemize}
		\item[(1)] $T(t)$ is chaotic on $L^{p}(\mathfrak{X})$.
		\item[(2)] $T(t)$ has a non-trivial periodic point, that is $L^{p}(\mathfrak{X})_{per}\neq\emptyset$.
		\item[(3)] $a$ and $b$ satisfy  $-\Re a-\Phi_{p}(a)<b<-\Re a+\Phi_{p}(a),$ where $\Phi_p(a)$ is given by (\ref{Phia}).
		\item [(4)] $T(t)$ is hypercyclic.
	\end{itemize}
\end{theoremC*}
\begin{Remark}
If we assume $b$ to be a complex number, then there won't be any significant change in the proof of the Theorem C because  $|e^{it\Im b}|=1$ for all $t\geq 0$. However there will be a minor modification in the statement of Theorem C (3) where $b$ will be replaced by $\Re b$.
\end{Remark}
	 As a consequence of Theorem C, we obtain the sharp range of perturbations for which the heat semigroup and the Schr\"{o}dinger semigroup are chaotic.
The heat semigroup$(\mathcal H_t)_{t\geq0}$ generated by the Laplacian $\mathcal L$ is given by the formula
$$\mathcal H_t=e^{-t\mathcal L}=\sum\limits_{n=0}^\infty\frac{(-t\mathcal L)^n}{n!}.$$
 In \cite{CMS} Cowling, Meda, and Setti studied the behaviour of $\mathcal H_t$ and its $L^p-L^r$ operator norm. Further Setti considered
the analogous problem for the complex-time heat operator $\mathcal H_\xi,$  and derived precise estimates for the $L^p-L^r$ operator norms of $\mathcal H_\xi,$ for $\xi$ belonging to the half-plane $\Re\xi\leq 0$ in his paper \cite{AS}. Proof of the Theorem C largely depends on the sharp $L^p$ norm estimates of the operator $e^{\xi\cL}$ for $\xi\in \C$, subsequently we have extended Setti's result for all $\xi\in \C$.

To put things in perspective, we now discuss the background of the subject. The study of chaotic dynamics of the heat semigroup on Riemannian symmetric spaces $M$ of non-compact type,  started with the work of Ji and Weber \cite{JW}. They studied the chaotic behaviour of certain shifts of the heat semigroup corresponding to the Laplace--Beltrami operator $\Delta$ endowed with Riemannian structure, namely $$T(t)=e^{-t(\Delta-c)}, \;\; t\geq0, \;\;\ c\in\R$$
 on the space of all radial functions on $L^p(M)$. In \cite{MP}, Pramanik and Sarkar extended the above result,  and gave a complete characterization for the chaotic behaviour  of the semigroup $T(t)$ on the whole of $L^{p}(M)$ and its related subspaces. A similar result concerning the chaotic behaviour of the heat semigroup (resp. Dunkl heat semigroup) is also known for harmonic $NA$ groups (resp. Euclidean spaces) (see \cite{BT} and \cite{RPS}). In \cite{MP}, the authors proved the following result:
 \begin{Theorem} Let $M$ be a Riemannian symmetric space of non-compact type, with $T(t)$ defined as above and $c_p=\frac{4|\rho|^2}{pp'}$. Then for $2<p<\infty$, $T(t)$ is chaotic on $L^{p}(M)$ if and only if $c>c_{p}$.
 \end{Theorem}
 The proof of the above theorem is largely influenced by the work of Desch, Schappacher and Webb. In \cite{DSW} Desch et al. gave a sufficient condition for a semigroup to be chaotic in terms of spectral properties of its generator. Hence the $L^{p}$-spectrum of $\Delta$ on $M$, which is a $p$-depending parabolic region, played a vital role in examining chaoticity and determining the range of perturbation for which the heat semigroup is chaotic. The $p$-dependence of the perturbation (that is, $c_{p}$), which relied heavily on the spectral properties of the generator, demanded that its point spectrum must contain infinitely many purely imaginary points.

A homogeneous tree may be viewed as a discrete analogue of a hyperbolic space. Nevertheless, there are certain differences due to the structural difference between these spaces. One such difference is that  the Laplacian $\mathcal L$ is a bounded operator on $L^p(\x)$. But, the  $L^p$-spectrum of the Laplacian $\mathcal L$ is also a $p$-depending conic region (in fact it is an elliptic region (\ref{lp-spectrum}) in $\C$). This elliptic region plays an important role in examining chaoticity and also in determining the range of perturbation for which the affine semigroup is chaotic.

Another motivation to study the chaotic dynamics of semigroups generated by analytic functions of the Laplacian on homogeneous trees is the work of deLaubenfels and Emamirad. In \cite{DE}, they studied the dynamical behavior of a class of operators, which is generated by non-constant analytic functions of the shift operator on weighted $L^p(\N)$ spaces.

We end this section by providing a quick outline on the contents of this article. We have collected all relevant notation, definitions, and facts about the homogeneous trees in Section 2. The proofs of Theorem A and Theorem B are given in Section 3. In Section 4 we prove Theorem C and discuss some of its notable consequences.

\section{Background materials on homogeneous trees}
\subsection{General Notations} The letters $\mathbb{R}$ and $\mathbb{C}$ will denote the set of all real numbers and complex numbers respectively. For $z\in \C$ we use the notation $\Re z$ and $\Im z$ for real and imaginary part of $z$ respectively. We will use the standard practice of using the letter $C$ for constant, whose  value may change from one line to another line. For every Lebesgue exponent $p\in(1,\infty)$, we write $p^{\prime}$ to denote the conjugate exponent $p/(p-1)$. Further we define $p^{\prime}=\infty$ when $p=1$ and vice-versa. For $p\in(1,\infty)$ let
$$\delta_p=\frac{1}{p}-\frac{1}{2}\;\;\; \text{and}\;\;\; S_p=\{z\in\C: |\Im z|\leq|\delta_p|\}.$$
We assume $\delta_1=-\delta_\infty=1/2$ so that $S_1=\{z\in\C: |\Im z|\leq1/2\}$. It is important to note that $\delta_p=-\delta_{p^{\prime}}$ and $S_{p}=S_{p^{\prime}}$ for any $p\in[1,\infty]$. When $p=2$, the infinite strip reduces to the real line. We shall henceforth write $S_p^\circ$ and $\partial{S_p}$ to denote the usual interior and the boundary of $S_p$ respectively. For any bounded linear operator $T$ defined on Lebesgue space $L^{p}(\mathfrak{X})$, we shall write $\sigma_p(T),~P(\sigma_{p}(T))$ to respectively denote the set of spectrum and point spectrum of $T$ in $L^{p}(\mathfrak{X})$.
\subsection{Basics}
Here we review some general facts about the homogeneous trees $\x$, most of which are already known (see, for e.g. \cite{CMS,FTC,FTP1,FTP2} and the references therein). Most of our notations are consistent with that of \cite{CMS}. Let $\x$ be a homogeneous tree and $d$ be the natural distance on $\x$. Let $o$ be a fixed reference point on $\x$. A function on $\x$ is said to be radial if $f(x)=f(y)$ whenever $d(x,o)=d(y,o)$.

Let $G$ be the group of isometries of the metric space $(\x,d)$ and $K$ be the stabilizer of $o$ in $G$. Then $\x$ can be realized as a coset space $G/K$ via the map $g\rightarrow g\cdot o$ and every function on $\x$ corresponds to a $K$-right invariant function on $G$. Further radial functions on $\x$ correspond to $K$-bi-invariant functions on $G$. Throughout this article we shall write $E(\x)^\#$ to denote the subspace of all radial functions in a function space $E(\x)$.
The boundary of $\x$, denoted by $\Omega$ is the set of all infinite geodesic rays of the form $\{\omega_{0},\omega_{1},\ldots\}$ where $\omega_{0}=o$ and $\omega_{n}\in\mathfrak{X}$. For fixed $\omega\in\Omega$, the map $k\rightarrow k\cdot\omega$ represents a transitive action of $K$ on $\Omega$.

 \subsection{Poisson Kernel and Poisson Transform} On the boundary $\Omega$, there exists a unique $K$-invariant, $G$-quasi-invariant probability measure $\nu$ and the Poisson kernel $p(g\cdot o,\omega)$ is defined to be the Radon-Nikodym derivative $d\nu(g^{-1}\omega)/d\nu(\omega)$. The Poisson kernel can be explicitly written as
\be \label{poissonkernel} p(x,\omega)=q^{h_{\omega}(x)} \;\; \forall x\in \x \;\; \forall \omega\in\Omega,\ee
where $h_{\omega}(x)$ is the height of $x$ with respect to $\omega$ (See \cite{FTC}). For $z\in\mathbb{C}$ and a suitable function $F$ defined on the boundary, its Poisson transformation $\mathcal{P}_{z}F$ is given by
$$\mathcal{P}_{z}F(x)=\int\limits_{\Omega}p^{1/2+iz}(x,\omega)F(\omega)d\nu(\omega).$$
It follows from the definition that $\mathcal{P}_{z}=\mathcal{P}_{z+\tau}$, where $\tau=2\pi/\log q$. It is also a well-known fact that $\mathcal L \mathcal{P}_{z}F(x) =\gamma(z) \mathcal{P}_{z}F(x)$ where $\gamma$ is an analytic function defined by the formula
\begin{equation}\label{gammaz}
\gamma(z)=1-\frac{q^{1/2+iz}+q^{1/2-iz}}{q+1}.
\end{equation}
\subsection{Spherical Function}
The elementary spherical function $\phi_z$ which is defined as $\mathcal{P}_z 1$, is given by the formula (see \cite{CMS,FTP1})
\begin{equation}\label{equation2.4}
\phi_z(x)= \begin{cases}
\vspace*{.2cm} \left(\frac{q-1}{q+1}|x|+1\right)q^{-|x|/2}&\forall z\in\ \tau\Z\\
\vspace*{.2cm}\left(\frac{q-1}{q+1}|x|+1\right)q^{-|x|/2}(-1)^{|x|}&\forall z\in {\tau/2}+\tau\Z\\
\mathbf{c}(z)q^{{(iz-1/2)}|x|}+\mathbf{c}(-z)q^{{(-iz-1/2)}|x|}&\forall z\in\C\setminus(\tau/2)\Z,
\end{cases}
\end{equation}
where $\mathbf{c}$ is a meromorphic function given by
$$\mathbf{c}(z)=\frac{q^{1/2}}{q+1}\frac{q^{1/2+iz}-q^{-{1/2}-iz}}{q^{iz}-q^{-iz}}\quad\forall z\in \mathbb{C}\setminus(\tau/2)\Z.$$
We now enlist some important properties of $\phi_{z}$ in the following lemma, most of which follows easily from the explicit formula above (for details see \cite{FTP1,KR}).
\begin{Lemma}\label{analytical phi}
	Let $\phi_{z}$ be the spherical function given by (\ref{equation2.4}). Then
	\begin{enumerate}
		\item[(i)] $\phi_{z}$ is a radial eigenfunction of $\mathcal{L}$ with eigenvalue $\gamma(z)$ and every radial eigenfunction of $\mathcal{L}$ with eigenvalue $\gamma(z)$ is a constant multiple of $\phi_{z}$.
		\item[(ii)] For every $x\in\mathfrak{X}$, the map $z\rightarrow\phi_{z}(x)$ is an entire function.
		\item[(iii)] $\phi_{z}\in L^{\infty}(\x)^\#$ if and only if $z\in S_{1}$.
        \item[(iv)] For $1<p<2$, $\phi_z\in L^{p^{\prime}}(\x)^\#$ if and only if $z\in S_p^\circ$.
        \item[(v)] For $1\leq p\leq 2$, $\phi_z\notin L^{p}(\x)$ for any $z\in S_p$.
\end{enumerate}
\end{Lemma}

We also need the following estimates of $\mathcal{P}_{z}F$, which can be considered as a generalisation of the size estimates of $\phi_{z}$ given above. For details about this topic we refer \cite{KR}.
 \begin{Proposition}\label{poisson estimates}
	Let $1<p<2$ and $z\in\mathbb{C}$ is such that $\Im z=\delta_{r^{\prime}}$ where $p<r<p^{\prime}$. Then for all $F\in L^{r^{\prime}}(\Omega)$,
	\begin{equation}\label{equation2.6}
	\|\mathcal{P}_{z}F\|_{L^{p^{\prime}}(\x)}\leq C \|F\|_{L^{r^\prime}(\Omega)}.
	\end{equation}
\end{Proposition}

\subsection{Spectrum of Laplacian}
Keeping in mind that chaoticity is largely dependent on the  $L^{p}$-spectrum of the Laplacian, we now recall some its important facts. The Laplacian $\mathcal{L}$ on $\mathfrak{X}$ is defined by
\be\label{laplacian} \mathcal Lf(x)=f(x)-\frac{1}{q+1}\sum\limits_{y:d(x,y)=1}f(y).\ee
It is easy to show that $\cL$ is a bounded operator from $L^p(\x)$ into itself.
It follows from Lemma \ref{analytical phi} above that for $p\in(2,\infty)$ and $z\in S_{p}^{\circ},$  $\phi_{z}$ are the $L^{p}$ -eigenfunctions of $\mathcal{L}$ and hence $\gamma(S^{\circ}_{p})$  lies inside the set $P(\sigma_p(\mathcal L))$ of $L^p$-point spectrum of $\mathcal{L}$. In fact $\gamma(S^{\circ}_{p})$ is exactly the set $P(\sigma_p(\mathcal L))$. To prove this, let us assume that there exists a non-zero function $u$ in $L^{p}(\mathfrak{X})$ such that $\mathcal{L}u=\gamma(z)u$ for some $z\notin S^{\circ}_{p}$. Suppose that $u(x_o)\neq 0$ for some $x_o\in \mathfrak X$. Then $f(x)=\int_Ku(x_okx)dk$ is a radial, eigenfunction of $\mathcal{L}$ with eigenvalue $\gamma(z)$ and hence by Lemma \ref{analytical phi}, $f$ is a constant multiple $\phi_z$. It is easy to show that $f\in L^p(\x)^\#$, which is clearly not possible since $\phi_{z}\notin L^{p}(\mathfrak{X})^\#$ for $z\notin S^{\circ}_{p}$ (by Lemma \ref{analytical phi}). In fact using Lemma \ref{analytical phi} and a similar technique as above, one can completely summarize the $L^p$-point spectrum of $\mathcal{L}$ as follows:
 \begin{Proposition}\label{lp-pointspectrum}
	Regarding the $L^{p}$-point spectrum of $\mathcal{L}$, we have the following results.
	\begin{enumerate}
		\item[(i)] For $1\leq p\leq2$, the point spectrum of $\mathcal{L}$ on $L^{p}(\mathfrak{X})$ is empty.
			\item[(ii)] For $2<p<\infty$, the point spectrum of $\mathcal{L}$ on $L^{p}(\mathfrak{X})$  is the set $\gamma(S^{\circ}_{p})$.
	\end{enumerate}
\end{Proposition}
The complete description of the $L^p$-spectrum of $\cL$ is given in the following Proposition. For details we refer
\cite{CMS,FTC,FTP2}.
  \begin{Proposition}\label{proposition2.6}
	For every $p\in[1,\infty]$, the $L^{p}$-spectrum $\sigma_{p}(\mathcal{L})$ of $\mathcal{L}$ is the image of $S_p$ under the map $\gamma$, which is precisely the set of all $w$ in $\mathbb{C}$ which satisfies
	\begin{equation}\label{lp-spectrum}
	\left[\frac{1-\Re(w)}{b\cosh(\delta_{p}\log q)}\right]^{2}+\left[\frac{\Im(w)}{b\sinh(\delta_{p}\log q)}\right]^{2}\leq 1,\text{ where }b=\frac{2\sqrt{q}}{q+1}.
	\end{equation}
	In particular, $\sigma_{2}(\mathcal{L})$ degenerates into the line segment $[1-b,1+b]$.
\end{Proposition}
The following elliptic region represents the $L^p$ spectrum of $\mathcal{L}$.
\begin{figure}[h]
\centering
\begin{tikzpicture}[scale=0.7]
\tikzset{edge/.style={decoration={markings,mark=at position 1 with {\arrow[scale=2,>=stealth]{>}}}, postaction={decorate}}}
\draw[edge] (0,0)--(8.5,0) node[below] {$X$};
\draw[edge] (4,-3)--(4,3) node[left] {$Y$};
\draw (0,2)--(8,2);
\draw (0,-2)--(8,-2);
\node[above right] at (4,2) {$|\delta_{p}|$};
\node[below left] at (4,-2) {$-|\delta_{p}|$};
\node[below left] at (4,0) {$O$};
\draw[edge] ([shift=(105:6)]11,-5) arc (105:75:6);
\draw[edge] (13.5,0)--(22,0) node[below] {$X$};
\draw[edge] (14.5,-3)--(14.5,3) node[left] {$Y$};
\draw (18,0) ellipse (3 cm and 2 cm);
\node[above] at (11,1) {$\gamma$};
\node at (6,1) {$S_p$};
\node[above] at (18,2) {$\gamma(S_p)$};
\node[above] at (18,0) {$(1,0)$};
\node[above right] at (15,0) {$\gamma(i|\delta_{p}|)$};
\node[below left] at (21,0) {$\gamma(\tau/2+i|\delta_{p}|)$};
\node[below left] at (14.5,0) {$O$};
\foreach \i in {15,21,18}{
\filldraw (\i,0) circle (2 pt);
}
\foreach \i in {2,-2}{
\filldraw (4,\i) circle (2 pt);
}
\end{tikzpicture}
\caption{$L^p$-spectrum of $\mathcal{L}$.}
\end{figure}
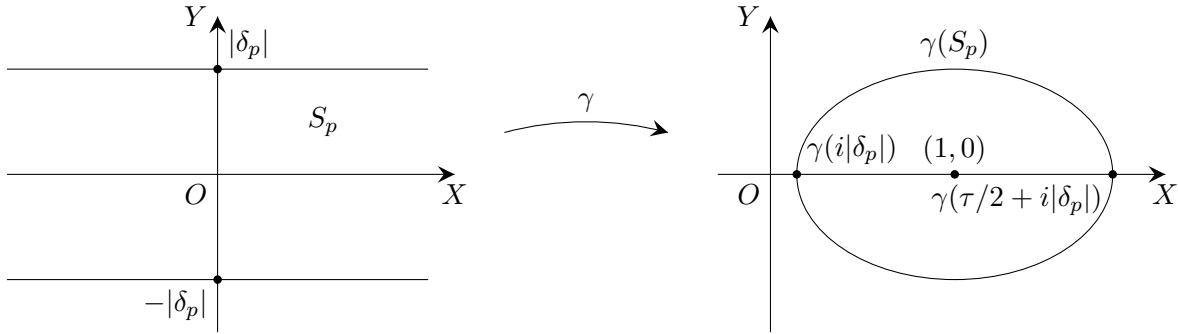

Note that for $2<p<\infty$, the open elliptic region in this figure also represents the $L^p$-point spectrum of $\mathcal{L}$ (i.e., $\gamma(S_{p}^\circ)$).
\subsection{The Helgason-Fourier Transform}
The Helgason-Fourier transform $\wtilde{f}$ of a finitely supported function $f$ is a function on $\mathbb \C\times\Omega$ defined by the formula
\begin{equation}\label{equation2.11}
	\wtilde{f}(z,\omega)=\sum\limits_{x\in\x}f(x)p^{1/2+iz}(x,\omega).
\end{equation}
For a finitely supported function $f$ on $\x$ and a continuous function $F$ on $\Omega$, we have
\begin{equation}\label{equation2.12}
\int\limits_{\Omega}\tilde{f}(z,\omega)F(\omega)d\nu(\omega)=\sum\limits_{x\in\mathfrak{X}}f(x)\left(\int_\Omega p^{1/2+iz}(x,\omega)F(\omega)d\nu(\omega)\right)=\sum\limits_{x\in\mathfrak{X}}f(x)\mathcal{P}_{z}F(x).
\end{equation}
Using the estimate of Poisson transform (see Proposition \ref{poisson estimates}) and the above duality relation, the authors of this paper proved the following (see \cite{KR}):

\begin{Theorem}\label{restriction}
 Let $1<p<2$ and $f\in L^{p}(\mathfrak{X})$. For $p<r<p'$ and $z\in\mathbb{C}$ with $\Im z=\delta_{r^{\prime}},$ there exists a constant $C_{p,r}>0$ such that
	\begin{equation}\label{equation2.13}
	\|\wtilde{f}(z,\cdot)\|_{L^{r}(\Omega)}\leq C_{p,r}\|f\|_{L^{p}(\mathfrak{X})}.
	\end{equation}
\end{Theorem}
The above expression is a `restriction type' inequality which shows that for each $z\in S_{p}^{\circ}$ with $z=\alpha+i\delta_{r^{\prime}}$, $\wtilde{f}(z,\cdot)$ exists as a measurable function in $L^{r}(\Omega)$. Regarding the analyticity of $\wtilde{f}(z,\cdot)$, we have the following result.

\begin{Lemma}\label{analytic} Suppose $f\in L^{p}(\mathfrak{X})$ for $1<p<2$. Then for every $\omega\in\Omega$, the map $z\rightarrow\wtilde{f}(z,\omega)$ is analytic on $S_{p}^{\circ}$.
	
\end{Lemma}

We conclude this section by providing the Plancherel Theorem which we shall require further.
\begin{Theorem}[Plancherel Theorem, \cite{FTP2}]\label{Plancherel}
	The Helgason--Fourier transform is an isometry from $L^{2}(\mathfrak{X})$ into $L^{2}([-\tau/2,\tau/2)\times\Omega,\mu\times\nu)$, that is, for $f\in L^{2}(\mathfrak{X})$,
	$$\|f\|^{2}_{L^{2}(\mathfrak{X})}=\int\limits_{-\tau/2}^{\tau/2}\int\limits_{\Omega}|\widetilde{f}(s,\omega)|^{2}d\nu(\omega)d\mu(s),$$
where $\mu$ denotes the Plancherel measure whose density with respect to the Lebesgue measure is given by $\tau^{-1}(q+1)^{-1}q/2|\mathbf c(s)|^{-2}$.
\end{Theorem}

\section{Proof of  Theorem A and Theorem B}
 To prove Theorem A and Theorem B,  we collect some key results which will be used very frequently.
As mentioned earlier, chaoticity of a semigroup is mainly triggered due to the abundance of its point spectrum. So, we need the following spectral mapping theorem (see \cite[Theorem 10.28, Theorem 10.33]{WR}).

\begin{Theorem}\label{theorem3.1}
	Suppose $T$ is a bounded linear operator on $L^{p}(\mathfrak{X})$ and $g$ is a non-constant complex holomorphic function defined on a connected open set containing $\sigma_{p}(T)$. Then we have the following.
	\begin{itemize}
		\item[(a)] $\sigma_{p}(g(T))=g(\sigma_{p}(T))$.
		\item[(b)] $P\sigma_{p}(g(T))=g(P\sigma_{p}(T))$.
	\end{itemize}
\end{Theorem}
For $1\leq p< \infty$ define
\be\label{origin} X_{0}=\{x\in L^p(\x):T(t)x\rightarrow 0\text{ as }t\rightarrow\infty\},\ee
\be \label{infinity} X_{\infty}=\{x\in L^p(\x):\forall\epsilon>0~\exists~w\in L^p(\x)\text{ and }t_{0}>0\text{ such that }\|w\|<\epsilon,\|T(t_{0})w-x\|<\epsilon\}.\ee
The following sufficient condition for hypercyclicity which was proved by Desch, Schappacher and Webb is useful in the sequel.
\begin{Proposition}\cite[Theorem 2.3]{DSW} \label{proposition2.2}
	Let $T(t)$, $t\geq 0$ be a strongly continuous semigroup on $L^p(\x)$ for $1\leq p< \infty$. If both the sets $X_0$ and $X_\infty$
	are dense in $L^p(\x)$, then $T(t)$ is hypercyclic.
\end{Proposition}


\subsection{Proof of Theorem A}	Fix $p\in(2,\infty)$. It follows from the definition that (1) implies (2) and (3). To complete the proof we only need to show that $(2)\implies (3)$ and $(3)\implies (1)$.

	\noindent We first prove  {(2)$\implies$(3):} For a clear understanding, we have divided this proof into the following steps.
	
	\noindent\textit{Step 1}:  In this step, we prove that $P\sigma_{p}(f(\mathcal{L}))\cap i\mathbb{R}$ is an infinite set. Condition (2) implies that there exists a nonzero function $h\in L^{p}(\mathfrak{X})$ such that $T(t_o)h=h$ for some $t_o>0,$ that is $1\in P\sigma_{p}(e^{t_of(\mathcal{L})})$. Using Proposition \ref{lp-pointspectrum} (ii) and Theorem \ref{theorem3.1} (b) we also have $P\sigma_{p}(e^{t_{0}f(\mathcal{L})})=e^{t_{0}f(\gamma(S_{p}^{\circ}))}$. Therefore there exists $z_{0}\in S_{p}^{\circ}$ such that $f(\gamma(z_{0}))=2n\pi i/t_o$ for some $n\in\mathbb{Z}$.

Let $\Gamma:S_{p}^{\circ}\rightarrow\mathbb{C}$ be defined by
	\begin{equation}\label{equation3.16}
	\Gamma(z)=(f\circ\gamma)(z)=f(\gamma(z)).
	\end{equation}
	It follows from the assumption on $f$ that $\Gamma$ is a non-constant holomorphic function on $S_{p}^{\circ}$.  Since $\Gamma (z_{0})=2n\pi i/t_o$ for some $z_{0}\in S_{p}^{\circ}$
and $ \Gamma(S_{p}^{\circ})= P\sigma_{p}(f(\mathcal{L}))$ therefore by the open mapping theorem it follows that $ P\sigma_{p}(f(\mathcal{L}))$ must contains some open ball centered at $2n\pi i/t_o$. Hence $P\sigma_{p}(f(\mathcal{L}))\cap i\mathbb{R}$ is an infinite set. In particular the set
	$V_{1}=\{z\in S_{p}^{\circ}:\Gamma(z)\in i\mathbb{Q}\}$
	is an infinite set which contains a cluster point in $S_{p}^{\circ}$.
	
	\noindent\textit{Step 2}: Let $z\in V_{1}$. Since $V_{1}\subseteq S_{p}^{\circ}$, therefore $z=\alpha+i\delta_{r^{\prime}}$ for some $r\in(p^{\prime},p)$ and $\alpha\in\mathbb{R}$. Hence the set
	$$\mathcal{V}_{1}=\bigcup\limits_{z\in V_{1}}\{\mathcal{P}_{z}F:F\in L^{r^{\prime}}(\Omega)\text{ whenever }\Im z=\delta_{r^{\prime}}\text{ with }p^{\prime}<r<p\}$$
	 is well defined.  It follows from inequality (\ref{equation2.6}) that $\mathcal{V}_{1}\subseteq L^{p}(\mathfrak{X})$. We now show that $\text{span}(\mathcal{V}_{1})\subseteq L^p(\x)_{per}$. Since $f(\mathcal{L})\mathcal{P}_{z}F=\Gamma(z)\mathcal{P}_{z}F$ for every $z\in S_{p}^{\circ}$, thus $T(t)\mathcal{P}_{z}F=e^{t\Gamma(z)}\mathcal{P}_{z}F$. Now if  $g\in\text{span}(\mathcal{V}_{1})$ then
	$$g=\sum\limits_{j=1}^{k}\beta_{j}\mathcal{P}_{z_{j}}F_{j},\text{ where }z_{j}\in V_{1}\text{ and }\beta_{j}\in\mathbb{C},1\leq j\leq k.$$
	Now for every $j\in\{1,2,\ldots,n\},$ $z_{j}\in V_{1}$ and hence , $\Gamma(z_{j})=ip_{j}/q_{j}$ such that $p_j,q_j\in \Z$ with $q_{j}\neq 0$.  If we choose $s=2\pi q_{1}\cdots q_{n}$, then $T(s)g=g$. Hence span($\mathcal{V}_{1})\subseteq L^p(\x)_{per}$.
	
	\noindent\textit{Step 3}: To prove that $\overline{L^{p}(\x)}_{per}=L^{p}(\mathfrak{X})$, it is enough to prove that span($\mathcal{V}_{1}$) is dense in $L^{p}(\mathfrak{X})$. Let $f\in L^{p^{\prime}}(\mathfrak{X})$ annihilates $\mathcal{V}_{1}$ that is $\sum\limits_{x\in\mathfrak{X}}f(x)\mathcal{P}_{z}F(x)=0$ for all $\mathcal{P}_{z}F\in\mathcal{V}_{1}.$
 Now by duality relation (\ref{equation2.12}) we have,

	$$\int\limits_{\Omega}\wtilde{f}(z,\omega)F(\omega)d\nu(\omega)=\sum\limits_{x\in\mathfrak{X}}f(x)\mathcal{P}_{z}F(x)=0\quad\text{ for all }\mathcal{P}_{z}F\in\mathcal{V}_{1}.$$
	Fix $z\in V_{1}$ and suppose that $z=\alpha+i\delta_{r^{\prime}}$ for some $r\in(p^{\prime},p)$. Then for every $F\in L^{r^{\prime}}(\mathfrak{X})$, we have
	$$\int\limits_{\Omega}\wtilde{f}(\alpha+i\delta_{r^{\prime}},\omega)F(\omega)d\nu(\omega)=0.$$
	Since $F\in L^{r^{\prime}}(\mathfrak{X})$ is arbitrary, therefore from Theorem \ref{restriction} and the above equation, we have $\wtilde{f}(\alpha+i\delta_{r^{\prime}},\omega)=0$ for almost every $\omega\in\Omega$.
  Thus for every $z\in V_{1}$, $\wtilde{f}(z,\omega)=0$ for almost every $\omega\in\Omega$. By the Lemma \ref{analytic}, for each $\omega\in\Omega$ the function $z\rightarrow\wtilde{f}(z,\omega)$ is analytic on $S_{p}^{\circ}$.
  So, we conclude that for almost every $\omega$, the set of zeros of $\wtilde{f}$ has a cluster point in $S_{p}^{\circ}$, and hence $\wtilde{f}(z,\omega)=0$ for every $z\in S_{p}^{\circ} $ and for almost every $\omega$.
  Since $f\in L^{p^{\prime}}(\mathfrak{X})\subseteq L^{2}(\mathfrak{X})$ (as $\mathfrak{X}$ is a discrete space) whenever $2<p<\infty$, therefore by  Plancherel Theorem \ref{Plancherel} we conclude that $f\equiv0$. This proves that span($\mathcal{V}_{1}$) is dense in $L^{p}(\mathfrak{X})$.

\noindent Now we will prove {(3)$\implies$(1):} Since the density of the periodic points is already assumed, so to prove our assertion it is enough to show that $T(t)$ is hypercyclic. In view of Proposition \ref{proposition2.2}, we only need to show that the sets $X_{0}$ and $X_{\infty}$ (as defined in Proposition \ref{proposition2.2}) are dense in $L^{p}(\mathfrak{X})$.
We define the sets
$$V_{2}=\{z\in S_{p}^{\circ}:\Re(\Gamma(z))<0\}\text{ and } V_{3}=\{z\in S_{p}^{\circ}:\Re(\Gamma(z))>0\}.$$
By repeating the arguments of \textit{Step 1} and \textit{Step 2} in the previous proof, one may prove that $V_2$ and $V_3$ are non-empty open sets and both contain a cluster point in $S_{p}^{\circ}$. Corresponding to the sets $V_{i}$, we define (for $i=2,3$)
$$\mathcal{V}_{i}=\bigcup\limits_{z\in V_{i}}\{\mathcal{P}_{z}F:F\in L^{r^{\prime}}(\Omega)\text{ whenever }\Im z=\delta_{r^{\prime}}\text{ with }p^{\prime}<r<p\}.$$
Adopting a similar approach as in \textit{Step 3}, we may easily show that both span($\mathcal{V}_2$)  and span($\mathcal{V}_3$) are dense in $L^p(\x)$.
The proof will be complete once we show that span($\mathcal{V}_{2})\subseteq X_{0}$ and span($\mathcal{V}_{3})\subseteq X_{\infty}$.
 For every $z\in V_{2}$,
$$\lim_{t\rightarrow\infty}\|T(t)\mathcal{P}_{z}F\|_{L^{p}(\mathfrak{X})}=\lim_{t\rightarrow\infty}e^{t\Re(\Gamma(z))}\|\mathcal{P}_{z}F\|_{L^{p}(\mathfrak{X})}=0.$$
This shows that $\mathcal{V}_{2}$ and hence span($\mathcal{V}_{2})$ is a subset of $X_{0}$.

Next we prove that span($\mathcal{V}_{3})\subseteq X_{\infty}$. Let $g\in\text{span}(\mathcal{V}_{3})$ be of the form
$$g=\sum\limits_{j=1}^{k}\alpha_{j}\mathcal{P}_{z_{j}}F_{j},\text{ where }z_{j}\in V_{3}\text{ and }\alpha_{j}\in\mathbb{C},1\leq j\leq k.$$
If we choose $g_t=\sum\limits_{j=1}^{k}e^{-t\Gamma(z_{j})}\alpha_{j}\mathcal{P}_{z_{j}}F_{j}$, then $T(t)g_t=g$ for all $t\geq 0$. Since $\Re(\Gamma(z_{j}))>0$ for each $j$, the limit $\|g_t\|_{L^{p}(\mathfrak{X})}\rightarrow 0$ as $t\rightarrow\infty$. Hence it follows from definition (\ref{infinity}) that span($\mathcal{V}_{3})\subseteq X_{\infty}$. This completes the proof.
\hfill\qed


\subsection{ Proof of Theorem B} {\it Part (1)}: Let $p\in[1,2]$. We know from Proposition \ref{lp-pointspectrum} (i), that $P\sigma_{p}(\mathcal{L})=\emptyset$ . Hence by Theorem \ref{theorem3.1} (b) it follows that $P\sigma_{p}(T(t))=\emptyset$ for all $t>0$. This show that only the zero function is a periodic point, that is, $T(t)$ has no non-trivial periodic point in $L^p(\x)$ for any $p\in[1,2]$.

{\it Part (2)}:  Now we will show that $T(t)$ is not hypercyclic on $L^{p}(\mathfrak{X})$. Let us first assume that $p=2$. We proof this assertion by contradiction. If possible, assume that  there exists a non-zero $h\in L^{2}(\mathfrak{X})$ such that the set $\{T(t)h:t\geq 0\}$ is dense in $L^{2}(\mathfrak{X})$. Then for $g=2h$, there exists a sequence of non-negative real numbers $\{t_{n}\}$ such that $T(t_{n})h\rightarrow g$ in $L^{2}(\mathfrak{X})$ as $n\rightarrow\infty$.

If the sequence $\{t_{n}\}$ is bounded, then there exists a subsequence $\{t_{n_{k}}\}$ of $\{t_{n}\}$ and some number $t_0\geq 0$ such that $t_{n_{k}}\rightarrow t_{0}$ as $k\rightarrow\infty$. Using the strong continuity of $T(t)$, it follows that $T(t_{n_{k}})h\rightarrow T(t_{0})h$ as $k\rightarrow\infty$. Hence $T(t_{0})h=g=2h$, which is impossible as $P\sigma_{2}(T(t_{0}))=\emptyset$.

 If the sequence $\{t_{n}\}$ is unbounded, then without the loss of generality we may assume that $\{t_n\}$ is strictly increasing to $\infty$. By using the Plancherel Theorem \ref{Plancherel}, we have
$$ 4\|h\|^{2}_{L^{2}(\mathfrak{X})}=\lim\limits_{n\rightarrow\infty}\|T(t_{n})h\|^{2}_{L^{2}(\mathfrak{X})}=
\lim\limits_{n\rightarrow\infty}\int\limits_{-\tau/2}^{\tau/2}\int\limits_{\Omega}\exp\{2t_{n}\Re f(\gamma(s))\}|\widetilde{h}(s,\omega)|^{2}d\nu(\omega)d\mu(s).$$

Let $S_{1}=\{s\in[-\tau/2,\tau/2):\exp\{2\Re f(\gamma(s))\}\leq1\}$ and $S_{2}=\{s\in[-\tau/2,\tau/2):\exp\{2\Re f(\gamma(s))\}>1\}$.
If the Plancherel measure of $S_2$ is zero then
$4\|h\|^{2}_{L^{2}(\mathfrak{X})}\leq\|h\|^{2}_{L^{2}(\mathfrak{X})}.$ This implies that $\|h\|_{L^{2}(\mathfrak{X})}=0,$ which is a contradiction to our assumption that $h\neq0$.
If $S_{2}$ has a positive Plancherel measure, then
$$4\|h\|^{2}_{L^{2}(\mathfrak{X})}\geq\lim\limits_{n\rightarrow\infty}\int\limits_{S_{2}}\int\limits_{\Omega}\exp\{2t_{n}\Re f(\gamma(s))\}|\widetilde{h}(s,\omega)|^{2}d\nu(\omega)d\mu(s).$$
By  the Monotone Convergence Theorem, it follows that the above integral tends to infinity as $n$ tends to infinity. This again leads to a contradiction. Hence we conclude that for any $h\in L^{2}(\mathfrak{X})$ and $h\neq 0$, the set $\{T(t)h:t\geq 0\}$ can never approximate $2h$. This proves our assertion for $p=2$.

Now we assume $p\in[1,2)$. Since $\mathfrak{X}$ is a discrete space, hence $L^{p}(\mathfrak{X})\subseteq L^{2}(\mathfrak{X})$ whenever $1\leq p<2$ and $\|h\|_{L^{2}(\mathfrak{X})}\leq \|h\|_{L^{p}(\mathfrak{X})}$ for every $h\in L^{p}(\mathfrak{X})$. This implies that for any $h\in L^{p}(\mathfrak{X})$ and $h\neq 0$, the set $\{T(t)h:t\geq 0\}$ can never approximate $2h$. This completes the proof.
\hfill\qed

\section{Proof of Theorem C and some of its consequences}

Before going into the details, we recall some important facts related to the operator $e^{\xi\mathcal{L}}$ for $\xi\in\C$. We refer \cite{AS} for a detailed study. The operator $e^{\xi\mathcal{L}}$ is a $G$-invariant, bounded operator on $L^{p}(\x)$. Let
$$h_{\xi}(x)=e^{-\xi}\sum\limits_{k=0}^{\infty}\frac{\xi^{k}}{k!}\mu^{(\ast k)}_{1},$$
and $\mu_{1}$ is the probability measure at a distance $1$ from the reference point $o$ and $\mu^{(\ast k)}_{1}$ denotes the $k$-th convolution power of $\mu_{1}$. Then  for $p\geq 1,$ $e^{\xi\mathcal{L}}f=f\ast h_{-\xi},$ for all $f\in L^p(\x)$. Note that we use a different parametrization, our $\gamma(z)$  corresponds to $1-\gamma(z)$ in \cite{AS}.
Now we prove the following Lemma, which gives the norm estimate of the operator $e^{\xi\mathcal{L}}$ for all $\xi\in \C$.

\begin{Lemma}\label{lemma4.1}
	Let $e^{\xi\mathcal{L}}$ be the operator  defined as above. Then for $ p >2$  the following hold.
	\be\label{norm_ezL}\exp\{\Re \xi+\Phi_{p}(\xi)\}\leq \|e^{\xi\mathcal{L}}\|_{p\to p}\leq C\exp\{\Re \xi+\Phi_{p}(\xi)\}.\ee
\end{Lemma}
\begin{proof}
For $\Re \xi\leq 0$, this result is already known (see \cite[Theorem 1]{AS}). Now we prove the result for $\Re \xi>0$. It was proved in \cite[Corollary 4]{AS} that for all non-zero $\xi$,
	\begin{equation}\label{equation3.14}
		\|h_{-\xi}\|_{L^{p,1}(\mathfrak{X})}\leq C\frac{\exp\{\gamma(0)\Re \xi\}}{|z|}\left(\sum\limits_{d=0}^{\infty}dq^{d\delta_{p}}|h^{\mathbb{Z}}_{-\xi(1-\gamma(0))}(d)|\right)
	\end{equation}
	where $h^{\mathbb{Z}}_{-\xi(1-\gamma(0))}(d)$ denotes the heat kernel associated to the heat operator on $\mathbb{Z}$. Here $\|h_{-\xi}\|_{L^{p,1}(\mathfrak{X})}$ denote the Lorentz $L^{p,1}$-norm of $h_{-\xi}$ (for details about Lorentz norm, we refer \cite{LG1}). It was proved in \cite[Page 748]{AS} that $h^{\mathbb{Z}}_{-\xi(1-\gamma(0))}(d)=e^{\xi(1-\gamma(0))}I_{d}(-\xi(1-\gamma(0)))$ where $I_{d}(\xi)$ denotes the modified Bessel function of order $d$.
Since $\Re\{-\xi(1-\gamma(0))\}<0$, so the arguments given in \cite{AS} cannot be applied here. However by  formula (8.01) from \cite[Page 379]{FWO}, we have $$h^{\mathbb{Z}}_{-\xi(1-\gamma(0))}(d)=e^{i\pi d}e^{2\xi(1-\gamma(0))}h^{\mathbb{Z}}_{\xi(1-\gamma(0))}(d)\quad\forall d\in\mathbb{N}\cup\{0\}.$$
Now by substituting the pointwise estimates of $h^{\mathbb{Z}}_{\xi(1-\gamma(0))}(d)$ from \cite[Lemma 6]{AS} in (\ref{equation3.14}), the result follows by just imitating the calculations given in that paper.

Now we will prove the lower bound of $\|e^{\xi\mathcal{L}}\|_{p\to p}$. For $p>2$,
$$\|e^{\xi\mathcal{L}}\|_{p\to p}\geq \frac{\|e^{\xi\mathcal{L}}\phi_{z}\|_{L^{p}(\x)}}{\|\phi_{z}\|_{L^{p}(\x)}}=\exp\{\Re{(\xi\gamma(z)})\}\;\; \text{for all}\; z\in S_{p}^{\circ}.$$
By taking the supremum over all $z\in S_{p}^{\circ},$ we have $\sup\{\exp\{\Re{(\xi\gamma(z)})\}:z\in S_{p}^{\circ}\}=\exp\{\Re\xi+\Phi_{p}(\xi)\}$. This gives the desired lower bound.
\end{proof}

Now we investigate the hypercyclicity of the semigroup $e^{(a\mathcal{L}+b)t}$ when $2<p<\infty$.

\begin{Lemma}\label{lemma4.2}
	Suppose that $T(t)=e^{t(a\mathcal{L}+b)}$ where $t\geq 0$, a is a non-zero complex number and $b$ is real. Then for $2<p<\infty$, $T(t)$ is not hypercyclic on $L^{p}(\mathfrak{X})$ whenever $b\leq -\Re a-\Phi_{p}(a)$ or $b\geq -\Re a+\Phi_{p}(a)$.
\end{Lemma}
\begin{proof} Fix $p\in(2,\infty)$ and let $a$ be a non-zero complex number. To prove that the semigroup $T(t)$ is not hypercyclic for the given range of $b$, it is enough to show that
the set $\{T(t)h:t\geq 0\}$ is not dense in $L^p(\x)$ for any $h\in L^p(\x)$. If $b\leq -\Re a-\Phi_{p}(a)$ then it follows from the Lemma \ref{lemma4.1}  that for every $h\in L^{p}(\mathfrak{X})$,
$$\|T(t)h\|_{L^{p}(\mathfrak{X})}=e^{bt}\|e^{at\mathcal{L}}h\|_{L^{p}(\mathfrak{X})}\leq C\exp\{t(b+\Re a+\Phi_{p}(a))\}\|h\|_{L^{p}(\mathfrak{X})}\leq C\|h\|_{L^{p}(\mathfrak{X})} .$$
This show that the set $\{T(t)h:t\geq 0\}$ is bounded and hence it cannot be dense.

 Now we consider the case when $b\geq -\Re a+\Phi_{p}(a)$. We prove this assertion by contradiction. Suppose that there exists a non-zero $h$ in $L^{p}(\mathfrak{X})$ such that $\{T(t)h:t\geq 0\}$ is dense in $L^{p}(\mathfrak{X})$. Note that for all $t\geq 0$ $e^{at\mathcal{L}}e^{-at\mathcal{L}}h=e^{-at\mathcal{L}}e^{at\mathcal{L}}h=h$. By using the norm estimate (\ref{norm_ezL}), we have
$$\|h\|_{L^{p}(\mathfrak{X})}=\|e^{-at\mathcal{L}}e^{at\mathcal{L}}h\|_{L^{p}(\mathfrak{X})}\leq C\exp\{t(-\Re a+\Phi_{p}(a))\}\|e^{at\mathcal{L}}h\|_{L^{p}(\mathfrak{X})}.$$
From the above inequality we obtain
$$\|T(t)h\|_{L^{p}(\mathfrak{X})}=e^{bt}\|e^{at\mathcal{L}}h\|_{L^{p}(\mathfrak{X})}\geq C\exp\{t(b+\Re a-\Phi_{p}(a))\}\|h\|_{L^{p}(\mathfrak{X})}\geq C\|h\|_{L^{p}(\mathfrak{X})} .$$
Thus we conclude that the function `0' does not belong to the closure of $\{T(t)h:t\geq 0\}$ in $L^{p}(\mathfrak{X})$ and we finally arrive at a contradiction. This shows that $\{T(t)h:t\geq 0\}$ cannot be dense in $L^{p}(\mathfrak{X})$ for any $h\in L^{p}(\mathfrak{X})$. This completes the proof.
\end{proof}

\subsection{Proof of Theorem C} The equivalence of the conditions (1) and (2) is already proved in Theorem A. Also (4) implies (3) is a consequence of the Lemma \ref{lemma4.2} and (1) implies (4) is obvious. Now to complete the proof we only need to show the equivalence of (3) and (1), which will follow from the the following lemma.
\begin{Lemma}\label{lemma4.3}
	Let $T(t)$ be defined as in Theorem C. Then for $2<p<\infty$, $T(t)$ is chaotic on $L^{p}(\mathfrak{X})$ if and only if
$ -\Re a-\Phi_{p}(a)< b<  -\Re a+\Phi_{p}(a)$.
\end{Lemma}
\begin{proof}Let $T(t)$ is chaotic on $L^p(\x)$. Then $L^p(\x)_{per}\neq\{0\}$. As in the proof of Theorem A there exists $z_{0}\in S_{p}^{\circ}$ such that $\exp\{(a\gamma(z_{0})+b)t_0)\}=1$ for some $t_0>0$. By applying the open mapping theorem on the analytic function $\exp\{(a\gamma(\cdot)+b)t_0\}$, it follows that the set $\exp\{(a\gamma(S_{p}^{\circ})+b)t_0\}$ contains an open ball centered at 1. This implies that
	\be \label{maxestimates} \max\limits_{z\in S_{p}}|e^{t_0(a\gamma(z)+b)}|>1\quad\text{and}\quad\min\limits_{z\in S_{p}}|e^{t_0(a\gamma(z)+b)}|<1.\ee
Suppose $a=x+iy$ and that $z=s+i\delta_{p}$. Let $h(s)=\Re(a\gamma(s+i\delta_{p})+b)$ be a complex-valued function defined for all $s\in[-\tau/2,\tau/2]$.
It follows from the definition of $\gamma$ that  $\Re \gamma(s+i\delta_{p})=1-\frac{q^{1/p}+q^{1/p^{\prime}}}{q+1}\cos(s\log q)\;\text{ and }\;\Im \gamma(s+i\delta_{p})=\frac{q^{1/p}-q^{1/p^{\prime}}}{q+1}\sin(s\log q).$ Hence
$$h(s)=x-\frac{q^{1/p}+q^{1/p^{\prime}}}{q+1}x\cos(s\log q)-\frac{q^{1/p}-q^{1/p^{\prime}}}{q+1}y\sin(s\log q)+b.$$
	A straightforward computation yield that the maximum and the minimum values of $h$ on the interval $[-\tau/2,\tau/2]$ are $x+\Phi_{p}(a)+b$ and $x-\Phi_{p}(a)+b$ respectively.
Since $\gamma$ is a $\tau$-periodic function, therefore by applying the Maximum Modulus principle on the function $e^{t(a\gamma(\cdot)+b)}$, we obtain
	$$ \max\limits_{z\in S_{p}}|e^{t(a\gamma(z)+b)}|=\exp\{(\Re a+\Phi_{p}(a)+b)t\} \;\; \text{and} \quad\min\limits_{z\in S_{p}}|e^{t(a\gamma(z)+b)}|=\exp\{(\Re a-\Phi_{p}(a)+b)t\}.$$
Hence it follows from the above estimates and  (\ref{maxestimates}) that $ -\Re a-\Phi_{p}(a)< b<  -\Re a+\Phi_{p}(a)$.

Now we prove the converse. It is follows from the above discussion that  if $-\Re a-\Phi_{p}(a)<b<-\Re a+\Phi_{p}(a)$, then for any fix $t>0$,
	$\max\limits_{z\in S_{p}}|e^{t(a\gamma(z)+b)}|>1\quad\text{and}\quad\min\limits_{z\in S_{p}}|e^{t(a\gamma(z)+b)}|<1.$ Now it can be proved easily that there exists
 $z_0\in S_{p}^{\circ}$ such that $|e^{t(a\gamma(z_0)+b)}|=1$. Therefore the assertion follows by a similar argument given in the proof of Theorem A.
This completes the proof. \end{proof}

\subsection{Some Consequences}
There are some well-known examples of semigroups which are generated by the affine functions. As a consequence of Theorem C we have the following interesting results about the chaotic dynamics of these semigroups.
\subsubsection{The Heat Semigroup} It is already mentioned in the introduction that the chaotic dynamics of the heat semigroup generated by certain shifts of the Laplace-Beltrami operator, are extensively studied on symmetric spaces \cite{JW,MP} and harmonic $NA$-groups \cite{RPS}. Our objective is to formulate these results for the heat semigroup on homogeneous trees by using Theorem C as a tool. Recall that the heat semigroup on $\mathfrak{X}$ generated by certain shifts of $\mathcal{L}$ is defined by the formula
$$T(t)=e^{-t(\mathcal{L}-b)}\quad\text{ where }t\geq 0,~b\in\mathbb{R}.$$
 Using Theorem B, it is clear that for any $b\in\mathbb{R}$, $T(t)$ is neither hypercyclic nor does it have any non-trivial periodic point on $L^{p}(\mathfrak{X})$ whenever $1\leq p\leq 2$. However putting $a=-1$, the following result is an immediate consequence of Theorem C.

\begin{Theorem}
	Suppose that $T(t)=e^{-t(\mathcal{L}-b)}$ where $t\geq 0$. Then for $2<p<\infty$, the following are equivalent.
	\begin{itemize}
		\item[(1)] $T(t)$ is chaotic on $L^{p}(\mathfrak{X})$.
		\item[(2)] $T(t)$ has a non-trivial periodic point.
		\item[(3)] $b$ satisfies the relation $\gamma(i\delta_{p})<b<\gamma(\tau/2+i\delta_{p})$.
		\item [(4)] $T(t)$ is hypercyclic.
	\end{itemize}
\end{Theorem}

The geometrical interpretation of the above result can also be seen in the following figure. The figure on the left represents the $L^p$-point spectrum of $\mathcal{L}$ and figure on right represents the $L^p$-point spectrum of $\mathcal{L}-b$. It is easy to see that the $L^p$-point spectrum of $\mathcal{L}-b$ cuts the imaginary axis at infinitely many points if and only if $\gamma(i\delta_{p})<b<\gamma(\tau/2+i\delta_{p})$.

\begin{figure}[h]
\centering
\begin{tikzpicture}[scale=0.7]
\tikzset{edge/.style={decoration={markings,mark=at position 1 with {\arrow[scale=2,>=stealth]{>}}}, postaction={decorate}}}
\draw[edge] (0,0)--(8.5,0) node[below] {$X$};
\draw[edge] (1,-2.5)--(1,2.5) node[left] {$Y$};
\draw (4.5,0) ellipse (3 cm and 2 cm);
\node[above right] at (1.5,0) {$\gamma(i\delta_{p})$};
\node[below left] at (7.5,0) {$\gamma(\tau/2+i\delta_{p})$};
\node[below left] at (1,0) {$O$};
\draw[edge] ([shift=(105:6)]11,-5) arc (105:75:6);
\draw[edge] (13.5,0)--(21.5,0) node[below] {$X$};
\draw[edge] (16,-2.5)--(16,2.5) node[left] {$Y$};
\draw (17.5,0) ellipse (3 cm and 2 cm);
\node[below left] at (14.5,0) {$\gamma(i\delta_{p})-b$};
\node[below left] at (20.5,0) {$\gamma(\tau/2+i\delta_{p})-b$};
\node[below left] at (16,0) {$O$};
\foreach \i in {1.5,7.5,14.5,20.5}{
\filldraw (\i,0) circle (2 pt);
}
\end{tikzpicture}
\caption{•}
\end{figure}

\subsubsection{The Schr\"{o}dinger Semigroup} Now we consider the Schr\"{o}dinger semigroup generated by the perturbation of $i\mathcal{L}$.
Once again using Theorem C with $a=i$, we have the following.

\begin{Theorem}\label{schrodinger}
	Suppose that $T(t)=e^{t(i\mathcal{L}+b)}$ where $t\geq 0$. Then for $2<p<\infty$, the following are equivalent.
	\begin{itemize}
		\item[(1)] $T(t)$ is chaotic on $L^{p}(\mathfrak{X})$.
		\item[(2)] $T(t)$ has a non-trivial periodic point.
		\item[(3)] $b$ satisfies the relation $\Im\gamma(\tau/4+i\delta_{p})<b<\Im\gamma(-\tau/4+i\delta_{p})$.
		\item [(4)] $T(t)$ is hypercyclic.
	\end{itemize}
\end{Theorem}

In a similar way as above, Theorem \ref{schrodinger} can also be described geometrically using the following figure.

\begin{figure}[H]
\centering
\begin{tikzpicture}[scale=0.7]
\tikzset{edge/.style={decoration={markings,mark=at position 1 with {\arrow[scale=2,>=stealth]{>}}}, postaction={decorate}}}
\draw[edge] (0,0)--(8.5,0) node[below] {$X$};
\draw[edge] (1,-2.5)--(1,2.5) node[left] {$Y$};
\draw (4.5,0) ellipse (3 cm and 2 cm);
\node[above] at (4.5,2) {$\gamma(-\tau/4+i\delta_{p})$};
\node[below] at (4.5,-2) {$\gamma(\tau/4+i\delta_{p})$};
\node[below left] at (1,0) {$O$};
\draw[edge] ([shift=(105:6)]10.5,-5) arc (105:75:6);
\draw[edge] (14,-4)--(14,4.5) node[left] {$Y$};
\draw[edge] (13,-3)--(19.5,-3) node[below] {$X$};
\draw (15.5,0.5) ellipse (2 cm and 3 cm);
\node[above right] at (17.5,0.5) {$i\gamma(\tau/4+i\delta_{p})+b$};
\node[rotate=40, below left] at (13,1) {$i\gamma(-\tau/4+i\delta_{p})+b$};
\node[below left] at (14,-3) {$O$};
\foreach \i/\j in {4.5/2,4.5/-2,17.5/0.5,13.5/0.5}{
\filldraw (\i,\j) circle (2 pt);
}
\end{tikzpicture}
\caption{•}
\end{figure}

 
\end{document}